\documentclass[12pt,a4paper,reqno]{amsart}
\usepackage[english]{babel}
\usepackage{amsmath}
\usepackage{amsfonts}
\usepackage{amsthm}
\usepackage{amssymb}
\usepackage{epsfig}
\usepackage{graphicx}
\usepackage{times}

\usepackage{latexsym}
\usepackage{xspace}

\usepackage[colorlinks=true,linkcolor=red,citecolor=blue]{hyperref}
\usepackage{accents}
\theoremstyle{plain}

\newtheorem{teo}{Theorem}[section]
\newtheorem*{ghys}{Theorem (Ghys)}
\newtheorem{pro}[teo]{Proposition}
\newtheorem{lem}[teo]{Lemma}
\newtheorem{cor}[teo]{Corollary}

\theoremstyle{definition}
\newtheorem{defi}[teo]{Definition}

\numberwithin{equation}{teo}

\theoremstyle{remark}
\newtheorem{rem}[teo]{Remark}
\newtheorem*{example*}{Example}

\begin{document}

\title[Nilpotent actions on the torus]
{Finite orbits for nilpotent actions on the torus}

\author{S. Firmo and J. Rib\'on}
\address{Instituto de Matem\'{a}tica e Estat\'\i stica \\
Universidade Federal Fluminense, Rua M\'{a}rio Santos Braga s/n -
Valonguinho, 24020\,-\,140 Niter\'{o}i, Rio de Janeiro - Brasil }
\email{firmo@mat.uff.br}
\email{javier@mat.uff.br}

\date{\today}
\thanks{Supported in part by CAPES}

\thispagestyle{empty}

\begin{abstract}

A homeomorphism of the \,$2$-torus with Lefschetz number different from zero has a fixed point. We give a
version of this result for nilpotent groups of diffeomorphisms. We prove that a  nilpotent group of \,$2$-torus
 diffeomorphims has finite orbits when the group has some element with Lefschetz number different from zero.

\end{abstract}

\maketitle

\thispagestyle{empty}

\vskip20pt
\section{Introduction}

 Abelian groups of isotopic to the identity  \,$C^{1}$-\,diffeomorphisms
of  closed orientable surfaces  different from the \,$2$-sphere and the \,$2$-torus 
have global fixed points \cite[Franks-Handel-Parwani]{fhp02}, i.e. there exists a common
fixed point for all elements of the group. They also prove that if
the surface is the \,$2$-sphere then such groups have  finite orbits with at most two elements
\cite{fhp01}. This result was generalized to nilpotent groups by the second
author \cite{JR:arxivsp}.
The situation is different for the \,$2$-torus ${\mathbb T}^{2}$ 
since there are isotopic to the identity diffeomorphisms  with no 
finite orbits. For instance consider the diffeomorphism \,$\tilde{\phi}: {\mathbb R}^{2} \to {\mathbb R}^{2}$\,
defined by \,$\tilde{\phi} (x,y) = (x + \sqrt{2},y)$. It induces a holomorphic diffeomorphism 
\,$\phi: {\mathbb R}^{2}/{\mathbb Z}^{2} \to {\mathbb R}^{2}/{\mathbb Z}^{2}$\, without finite orbits.

In this context  when the surface is different from the \,$2$-torus the existence of global fixed points 
or finite orbits is essentially imposed by the surface topology
and the dynamics of the nilpotent (or abelian)  relation on the group.
Several papers have focused on this issue as 
\cite{elon02,pl01} for abelian and nilpotent connected Lie groups respectively, 
\cite{bo01,bo02,han01,fir01,fhp01,fhp02,fir03} for abelian groups and
\cite{fir02,JR:arxivsp,rib01} for the nilpotent ones.

In order to find finite orbits for the \,$2$-torus \,$\mathbb{T}^{2}$\, we need to impose more conditions 
on the groups other than the nilpotent property.
The conditions can be either of topological nature or of more dynamical type.
In this article we obtain finite orbits through the former approach.
The latter point of view
is studied in the forthcoming article
\,{\it Global fixed points for nilpotent actions on the torus} (cf. \cite{fr01}).

A natural topological condition on
\,$G \subset \mathrm{Diff}^{1}(\mathbb{T}^{2})$\, to obtain finite orbits is the existence of some element
in \,$G$\, whose Lefschetz number is different from zero.
Such property plays the role of a rigidity condition on the nilpotent group $G$.
It allows to show the following result.

\vskip10pt
\begin{teo}
\label{teo:main5}
Let \,$G$\, be a nilpotent subgroup of \,$\mathrm{Diff}^{1}({\mathbb T}^{2})$\,.
If \,$G$\, has some element whose  Lefschetz number
is different from zero
then \,$G$\, has a finite orbit.
\end{teo}

 Here, \,$\mathrm{Diff}^{1}({\mathbb T}^{2})$\, denotes the set of
\,$C^{1}$-diffeomorphisms of the \,$2$-torus.

The Lefschetz fixed point theorem guarantees a fixed point for each element of
\,$G$\, whose Lefschetz number is different from zero.
But the existence of such an element
is not sufficient to guarantee a global fixed point  since for instance
\,$G$\, may have some elements without fixed points.
This is the case for  the abelian subgroup \,$G$\, of
\,$\mathrm{Diff}^{1}(\,\mathbb{S}^{1} \! \times \mathbb{S}^{1})$\, generated by the maps
$$\psi(z_{1},z_{2})=(\bar{z}_{1},\bar{z}_{2}) \quad \text{and} \quad
\phi(z_{1},z_{2})=(-z_{1},-z_{2}) \quad
\text{where} \quad  z_{1},z_{2} \in \mathbb{S}^{1}\subset \mathbb{C}.$$
In this example \,$\phi$\, has no fixed point  and the Lefschetz number of \,$\psi$\, is equal
to \,$4$.

 Theorem \ref{teo:main5} provides a natural version, in the torus \,$\mathbb{T}^{2}$\, and for the \,$C^{1}$-\,differentiability  class, of some result proved by Ghys for the 2-sphere in the analytic case. In \cite{ghys93} Ghys proves the following theorem.

\vskip10pt
\begin{ghys}

Nilpotent groups of real analytic diffeomorphisms of \,$\mathbb{S}^{2}$ have  finite orbits.

\end{ghys}
\vskip5pt

Naturally, in the result of Ghys, the identity map is an element with Lefschetz number different from zero.
The arguments in the proofs of these two results are very different and 
it is not clear at all how to generalize the real analytic arguments to the $C^{1}$
case for \,$\mathbb{S}^{2}$\, or other closed 
orientable surfaces with Euler characteristic different from zero.

\vglue5pt
At the end of this article we present versions of Theorem \ref{teo:main5} for the cases 
where the surface is the Klein bottle, the compact annulus and the compact M\"obius strip.

\vglue5pt
Given a homeomorphism  \,$\psi$\, of \,${\mathbb T}^{2}$\,, 
we denote by \,$L(\psi)$\, the Lefschetz
number of \,$\psi$.
 Let us remark that the class of \,$\psi$\, in the mapping
class group of the torus
can be identified with a matrix \,$[\psi]$\, in \,$\mathrm{GL}(2\,,{\mathbb Z})$.
The following elementary property will be key in the proof of Theorem \ref{teo:main5}.
\vskip10pt
\begin{lem}\label{Lefs:spec}
The Lefschetz number of a homeomorphism \,$\psi$\, of \,$\mathbb{T}^{2}$\, is different from zero
if and only if \,$1 \notin \mathrm{spec}[\psi]$\,.
\end{lem}
\begin{proof}
Since
\,$L(\psi) = \det (Id - [\psi])$\, for tori in any dimension
(cf. \cite{Brooks-Brown-Pak-Taylor}), we conclude that
\,$L(\psi)$\, vanishes if and only if \,$1 \in \mathrm{spec}[\psi]$.
\end{proof}
The condition \,$1 \notin \mathrm{spec}[\psi]$\,
induces
some rigidity phenomena for \,$\psi$\, since in such a case the identity map is the unique deck 
transformation that commutes with \,$\tilde{\psi}$\, (cf. Lemma \ref{lem:pact}) 
where \,$\tilde{\psi}$\, is a lift of \,$\psi$\, to
the universal covering \,${\mathbb R}^{2}$\, of \,${\mathbb T}^{2}$. In particular the set 
\,$\mathrm{Fix}(\tilde{\psi})$\, of fixed points of \,$\tilde{\psi}$\, will be an (eventually empty) compact set  
by Lemma \ref{lim:fix:spec}.

\vglue5pt
 
 In order to show Theorem \ref{teo:main5} it suffices to find
a finite index normal subgroup of \,$G$\, that has a global fixed point.
For this we choose a convenient \,$\psi \in G$\, with \,$L(\psi)\neq 0$\, and a lift 
\,$\tilde{\psi}$\, such that \,$\mathrm{Fix}(\tilde{\psi}) \neq \emptyset$. We use
the property \,$1 \notin \mathrm{spec}[\psi]$\, and the
description of the nilpotent subgroups of
\,$\mathrm{GL}(2\,,{\mathbb Z})$\, 
to obtain a nilpotent subgroup \,$\tilde{H}$\, of diffeomorphisms of \,${\mathbb R}^{2}$\, consisting of 
lifts of elements of a finite index normal subgroup \,$H$\, of \,$G$. Moreover we can suppose    
\,$\tilde{\psi} \in \tilde{H}$. Since \,$\mathrm{Fix}(\tilde{\psi})$\, is a non-empty compact set, 
\,$\tilde{H}$\,
has a global fixed point by Theorem 1.4 of \cite{JR:arxivsp} (cf. Theorem \ref{cor:plane}).
Hence \,$H$\, has a global fixed point.

\vskip30pt
\section{Preliminaries}
\label{sec:pre}

This section is devoted to introducing some definitions and notations. For later reference, we also 
present the classification of  nilpotent subgroups of the mapping class group of \,${\mathbb T}^{2}$.
\begin{defi}
In the rest of this article  \,$\text{Fix}(f)$\,
denotes the set of fixed points of the map \,$f$. If \,$L$\, is a family of maps we note
\,$\text{Fix}(L):=\cap_{f\in L} \text{Fix}(f)$. We say that \,$\text{Fix}(L)$\, is the set of {\it global fixed points}
of the family \,$L$.
\end{defi}
Let \,$G$\, be a group and \,$H$\, be a subgroup of \,$G$. We denote by \,$[\,H,G\,]$\, the subgroup of
\,$G$\, generated by the elements of the form \,$[\,h\,,g\,]=hgh^{-1}g^{-1}$\, where
\,$h\in H$\, and \,$g\in G$. If \,$H$\, is a normal subgroup of \,$G$\, then \,$[\,H,G\,]$\,
is a  subgroup of \,$H$\, which is normal in \,$G$.

Given a group \,$G$\,  let us consider  the \,{\it upper central series}\,
\,$\{Z^{(n)}(G)\}_{n\geq0}$\, of \,$G$\,
$$Z^{(n+1)}(G):= \big\{g\in G \ \,; \ [\,g\,,f\,]\in Z^{(n)}(G) \ \ \text{for all} \ \ f\in G \big\}$$
where \,$Z^{(0)}(G)$\, is the trivial subgroup of \,$G$.
The members of the upper central series are characteristic subgroups of \,$G$. In particular they are normal
subgroups of \,$G$\, and we have
$$Z^{(0)}(G)\subset Z^{(1)}(G)\subset \cdots \subset Z^{(n)}(G) \subset \cdots \subset G .$$
If \,$Z^{(n)}(G)=G$\, for some \,$n\in \mathbb{Z}_{\geq0}$\, we say that \,$G$\, is a
\,{\it nilpotent group}. The smallest \,$n \in \mathbb{Z}_{\geq0}$\,
 such that \,$Z^{(n)}(G)=G$\, is the \,{\it nilpotency class}\, of \,$G$.

\vskip5pt

We denote by $\mathrm{Homeo}(M)$ and $\mathrm{Homeo}_{0}(M)$ the group 
of homeomorphisms of a manifold $M$ and 
its subgroup of homeomorphisms isotopic to the identity map respectively.


\vskip10pt
\begin{defi}
\label{def:imcg}
Let \,$[\,]: \mathrm{Homeo}({\mathbb T}^{2}) \to \mathrm{MCG}({\mathbb T}^{2})=
\mathrm{GL}(2\,,{\mathbb Z})$\,
be the map associating to an element
of \,$\mathrm{Homeo}({\mathbb T}^{2})$\, its image in the mapping class group.
Given a subgroup \,$G$\, of \,$\mathrm{Homeo}({\mathbb T}^{2})$\, we denote by \,$[G]$\,
the image of \,$G$\, by $[\,]$.
\end{defi}
\vskip5pt

Nilpotent subgroups of \,$\mathrm{Homeo}({\mathbb T}^{2})$\, induce nilpotent
subgroups of the mapping class group of \,${\mathbb T}^{2}$, i.e.
nilpotent subgroups of \,$\mathrm{GL}(2\,,{\mathbb Z})$. We will need a classification
of such groups in order to study rotational properties.
They are virtually cyclic and metabelian.
Moreover, there exists a unique example of a non-abelian group, up to conjugacy.

\vskip10pt
\begin{lem}
\label{lem:nmcg}
Let \,${\mathcal G}$\, be a nilpotent subgroup of
\,$\mathrm{MCG}({\mathbb T}^{2})$\,.
Then \,${\mathcal G}$\, is either of the form
\,$\langle N \rangle$\, or
\,$\langle N \,, -N \rangle$\, for some \,$N\in \mathcal{G}$\,,
or it is conjugated by a matrix in
\,$\mathrm{GL}(2\,,{\mathbb Q})$\, to the group
\[    {\mathcal H}:= \left\{
\left(
\begin{array}{rr}
1 & 0 \\
0 & 1 \\
\end{array}
\right), \
\left(
\begin{array}{rr}
-1 & 0 \\
0 & -1 \\
\end{array}
\right), \
\left(
\begin{array}{rr}
1 & 0 \\
0 & -1 \\
\end{array}
\right), \
\left(
\begin{array}{rr}
-1 & 0 \\
0 & 1 \\
\end{array}
\right), \right. \]
\[ \left.
\left(
\begin{array}{rr}
0 & -1 \\
1 & 0 \\
\end{array}
\right), \
\left(
\begin{array}{rr}
0 & 1 \\
-1 & 0 \\
\end{array}
\right), \
\left(
\begin{array}{rr}
0 & 1 \\
1 & 0 \\
\end{array}
\right), \
\left(
\begin{array}{rr}
0 & -1 \\
-1 & 0 \\
\end{array}
\right)
\right\}. \]
\end{lem}

The group ${\mathcal H}$ is isomorphic to the dihedral group $D_{4}$.
We are admitting orientation-reversing classes in the mapping class group.
Notice that if
all classes are orientation-preserving then \,${\mathcal G}$\, is abelian.
We did not find a proof of the above lemma in the literature and for sake of clarity 
we will prove it in Appendix A.

\vskip5pt
 
\noindent
 {\it Conventions}. From now on, we make the following conventions. A homeomorphism
\,$\tilde{\psi}\in \text{Homeo}(\mathbb{R}^{2})$\, always denotes a lift to the universal covering of
\,${\psi}\in \text{Homeo}(\mathbb{T}^{2})$\, and vice-versa. Moreover,
\,$\pi:\mathbb{R}^{2}\rightarrow\mathbb{T}^{2}$\, denotes the universal covering map and
unless explicitly stated otherwise a lift means a lift to the universal covering.

\vglue5pt

 Let \,$G$\, be a subgroup of \,$\mathrm{Homeo}(\mathbb{T}^{2})$.  We say that
a subgroup \,$\tilde{G}$\, of  \,$\mathrm{Homeo}(\mathbb{R}^{2})$\,
is a \,{\it lift}\, of \,$G$\,  if any element \,$\tilde{\phi}$\, of \,$\tilde{G}$\,
is a lift of  some element \,$\phi$\, of \,$G$\, and
the  natural projection \,$\kappa: \tilde{G} \to G$\, defined by \,$\kappa\big(\tilde{\phi}\big)=\phi$\, is an
isomorphism of groups.
Let us remark that the definition of lift for groups is more restrictive than the definition for single
homeomorphisms.
The translation \,$T_{(0,1)}$\, is a lift of the identity map
of \,${\mathbb T}^{2}$\, but the group \,$\langle T_{(0,1)} \rangle$\, is not a lift of
the group \,$\{Id\}$.

\vskip5pt

Now we introduce a result proved by the second author
in \cite{JR:arxivsp}
that will be used to find global fixed points
of nilpotent groups of diffeomorphisms of the torus.

\begin{teo}
\label{cor:plane}
Let \,$G$\, be a
nilpotent subgroup of \,$\mathrm{Diff}_{+}^{1}({\mathbb R}^{2})$\, such that
\,$\mathrm{Fix}(\phi)$\, is a non-empty compact set for some \,$\phi \in G$.
Then \,$G$\, has a global fixed point.
\end{teo}
\vskip5pt

In the above theorem, \,$\mathrm{Diff}^{1}_{+}(\mathbb{R}^{2})$\, denotes the set of  \,$C^{1}$ orientation preserving  diffeomorphisms of \,$\mathbb{R}^{2}$.

\vskip30pt
\section{Rotational properties}
\label{sec:rot}
 
Let us introduce rotation vectors for $\phi \in \mathrm{Homeo}_{0}({\mathbb T}^{2})$.
Following \cite{mi03}, the rotation vectors
of a lift \,$\tilde{\phi}$\, are the limits of  sequences of the form
\[ \frac{\tilde{\phi}^{\,n_{k}}(\tilde x_{k}) - \tilde x_{k}}{n_{k}} \]
where \,$(n_{k})_{k \geq 1}$\, is  an increasing sequence  of positive integers and
\,$(\tilde x_{k})_{k \geq 1}$\, is a sequence of points in \,$\mathbb{R}^{2}$.
This set will be denoted by \,$\rho\big(\tilde\phi\big)$. We know from
\cite{mi03} that it is a non-empty compact and convex subset of \,$\mathbb{R}^{2}$.

Equivalently we know that
\,$\rho \big(\tilde{\phi}\big) = \big\{ \rho_{\mu} \big(\tilde{\phi}\big) \ \, ; \ \mu \in {\mathcal P}(\phi) \big\}$\,
where \,${\mathcal P}(\phi)$\, is the set of $\phi$-invariant Borel probability measures and

$$ {\rho}_{\mu}\big(\tilde{\phi}\big): = \int_{{\mathbb T}^{2}} \big(\tilde{\phi} - Id \big)
 \ d \mu \,. $$
Notice that since \,$\phi$\, belongs to \,$\mathrm{Homeo}_{0}({\mathbb T}^{2})$\, then
\,$\tilde{\phi}$\, commutes with the covering transformations and
\,$\tilde{\phi}-Id$\, descends to a well-defined map in \,${\mathbb T}^{2}$.
The set \,$\rho \big(\tilde{\phi}\big)$\,  depends
on the lift \,$\tilde{\phi}$\, of \,$\phi$\, but it satisfies
\,$\rho \big(T_{v} \circ \tilde{\phi}\big) = T_{v}\big(\rho \big(\tilde{\phi}\big)\big)$\, where
 \,$T_{v}$\, is the  translation in \,$\mathbb{R}^{2}$\, by the vector
 \,$v\in\mathbb{Z}^{2}$.
In particular the projection \,$\rho (\phi)$\, of \,$\rho \big(\tilde{\phi}\big)$\, in
\,${\mathbb T}^{2} = {\mathbb R}^{2}/{\mathbb Z}^{2}$\, 
depends on \,$\phi$\, but it does not depend on the lift \,$\tilde{\phi}$\, of \,$\phi$.
In what follows we will be using frequently the next two well known results.

\begin{lem}
\label{lem:pact}
Let \,$\phi \in \mathrm{Homeo}_{0}({\mathbb T}^{2})$\,, \,$\mu\in\mathcal{P}(\phi)$\, and
\,$\psi \in \mathrm{Homeo}({\mathbb T}^{2})$\,.
For all lifts to the universal covering \,$\tilde{\phi}$\, and \,$\tilde{\psi}$\,  of \,$\phi$\, and \,$\psi$\,  respectively,
we have\,:
\[\tilde{\psi} \circ T_{v}= T_{[\psi](v)} \circ \tilde{\psi} \quad \text{and} \quad
[\psi]\Big(\rho_{\mu}\big(\tilde{\phi}\big)\Big) =
\rho_{\nu}\big(\tilde{\psi} \circ \tilde{\phi} \circ \tilde{\psi}^{-1}\big)  \]
where \,$v\in\mathbb{Z}^{2}$\, and
\,$\nu=\psi_{*}(\mu)$\,.
\end{lem}

\vskip5pt

\begin{lem}
\label{lem:mor}
Consider the subgroup \,$\mathrm{Homeo}_{0,\mu}({\mathbb T}^{2})$\,
of \,$\mathrm{Homeo}_{0}({\mathbb T}^{2})$\, whose elements preserve a
probability measure \,$\mu$\, and let
\,$\mathrm{\widetilde{H}omeo}_{0,\mu}({\mathbb R}^{2})$\,
be the subgroup of \,$\mathrm{Homeo}_{0}({\mathbb R}^{2})$\, consisting of all the lifts of 
elements of \,$\mathrm{Homeo}_{0,\mu}({\mathbb T}^{2})$. Then the map
\,$\rho_{\mu}: \mathrm{\widetilde{H}omeo}_{0,\mu}({\mathbb R}^{2}) \to {\mathbb R}^{2}$\,
is a morphism of groups.
\end{lem}
The proof is obtained by a change of variable argument.

Let \,$G$\, be a subgroup of \,$\mathrm{Homeo}({\mathbb T}^{2})$\,. We denote by \,$G_{0}$\,
the subgroup  of isotopic to the identity elements of \,$G$. 
 By \,${\mathcal P}(G)$\,  we denote
the set \,$\cap_{\psi \in G} \, {\mathcal P}(\psi)$\, of \,$G$-invariant Borel probability measures.
We say that an element
\,$\phi \in G_{0}$\, is \,$\mathcal{P}(G)$-{\it irrotational}\, if
\,$\mathcal{P}(G)\neq\emptyset$\, and
 there exists a lift
\,$\tilde{\phi} \in \mathrm{Homeo}_{0}(\mathbb{R}^{2})$\, of \,$\phi$\, such that
\,$\rho_{\mu}(\tilde{\phi})=(0\,,0)$\, for all \,$\mu \in {\mathcal P} (G)$.

 We define \,$G_{\mathcal{I}}$\, as 
 the set of all the elements of \,$G_{0}$\, that are
\,$\mathcal{P}(G)$-{\it irrotational}\, i.e.,
\[ G_{\mathcal I} := \big\{ \phi \in G_{0} \ \ ; \ \ \exists \ \ \text{a lift} \ \
\tilde{\phi} \ \ \text{s.t.} \ \ \rho_{\mu}(\tilde{\phi})=(0\,,0) \ \
\text{for all} \ \ \mu \in {\mathcal P} (G) \big\}. \]
Moreover, it follows from Lemmas \ref{lem:pact} and \ref{lem:mor}
that \,$G_{\mathcal{I}}$\, is a normal subgroup of
\,$G$\,  if \,${\mathcal P} (G)$\, is non-empty.


Now we introduce some new notations.
Suppose given \,$\mu \in \mathcal{P}(G)$.
Following the convention about \,$\tilde{\phi}$\, and \,$\phi$\,
we define the following sets\,:
\begin{align}
G^{\mu}_{\mathcal{I}}:= & \big\{ \phi \in G_{0} \ \ ; \ \exists \ \ \text{a lift} \
\ \tilde{\phi} \ \ \text{s.t.} \ \ \rho_{\mu}\big(\tilde{\phi}\big)=(0\,,0)  \big\}\,;    \notag   \\
\tilde{G}^{\mu}_{\mathcal{I}}:=  &  \big\{ \tilde{\phi} \in \text{Homeo}_{0}(\mathbb{R}^{2}) \ \ ; \
\phi\in G_{0} \ \ \text{and} \ \  \rho_{\mu}\big(\tilde{\phi}\big)=(0\,,0) \big\}\,;   \notag   \\
\tilde{G}_{\mathcal{I}}:=  &  \big\{ \tilde{\phi} \in \text{Homeo}_{0}(\mathbb{R}^{2}) \ \ ; \ \phi\in G_{0} \ \ \text{and}
\ \  \rho_{\nu}\big(\tilde{\phi}\big)=(0\,,0)   \notag   \ \
\text{for all} \ \  \nu \in {\mathcal P} (G) \big\}\,;        \notag \\
\notag
\end{align}
and we have
\begin{align}\label{rel:global}
\tilde{G}_{\mathcal{I}} = \bigcap_{\nu\in\mathcal{P}(G)} \!\! \tilde{G}^{\nu}_{\mathcal{I}} \,.
\end{align}

Since
\,$\rho_{\mu}:\mathrm{\widetilde{H}omeo}_{0,\mu}(\mathbb{R}^{2})\rightarrow\mathbb{R}^{2}$\,
is a morphism of groups,
\,$\tilde{G}^{\mu}_{\mathcal{I}}$\, and \,$\tilde{G}_{\mathcal{I}}$\,
\big(resp. \,$G^{\mu}_{\mathcal{I}}$\,  and \,${G}_{\mathcal{I}}$\big) are subgroups 
of \,$\text{Homeo}_{0}(\mathbb{R}^{2})$\,
\big(resp. \,$\text{Homeo}_{0}(\mathbb{T}^{2})$\big). Moreover, we have that
\,$\tilde{G}_{\mathcal{I}}$\,   and    \,$\tilde{G}^{\mu}_{\mathcal{I}}$\,
are lifts of
\,${G}_{\mathcal I}$\, and
\,${G}_{\mathcal I}^{\mu}$\, respectively, since the natural projections
\,$\tilde{\phi}\in\tilde{G}_{\mathcal I}^{\mu} \xrightarrow{\kappa} \phi \in {G}_{\mathcal I}^{\mu} $\, and
\,$\tilde{\phi}\in \tilde{G}_{\mathcal I} \xrightarrow{\kappa} \phi \in {G}_{\mathcal I}$\,  are isomorphisms.

Clearly, we also have  \,$G_{\mathcal{I}} \subset
G_{\mathcal{I}}^{\mu}$\, for any \,$\mu\in\mathcal{P}(G)$\,.
Notice that ${\mathcal P}(G)$ is non-empty
if $G$ is an amenable group.
In particular ${\mathcal P}(G)$ is non-empty
if $G$ is a nilpotent group.

In the remainder of this section we will see that
the $\mathcal{P}(G)$-irrotational subgroup \,$G_{\mathcal{I}}$\, of a nilpotent group
\,$G \subset \mathrm{Homeo}({\mathbb T}^{2})$\,
is well-behaved with respect to lifts.
 In particular, if \,$G$\, has an element with Lefschetz number different from zero then
we can even show the existence of a
finite $G$-orbit (cf. Theorem \ref{teo:fin}).

 Our goal in the remaining of this section is relating
the rotational properties of a nilpotent group with
the existence of lifts.

\vskip10pt
\begin{lem}
\label{lem:lifr}
Let \,$G$\, be a subgroup of \,$\mathrm{Homeo}({\mathbb T}^{2})$\,
and let \,$\widehat{G}\subset \mathrm{Homeo}(\mathbb{R}^{2})$\, be the subgroup of
all lifts of elements of \,$G$.
If \,$\mu \in \mathcal{P}(G)$\, then
\,$\tilde{G}_{\mathcal I}^{\mu}\,,\tilde{G}_{\mathcal I}$\, and
\,${G}_{\mathcal I}^{\mu}\,,{G}_{\mathcal I}$\, are normal subgroups of \,$\widehat{G}$\, and \,$G$\, respectively.

\end{lem}

\begin{proof}
Let \,$\tilde{\phi} \in \tilde{G}_{\mathcal I}^{\mu}$\, and \,$\tilde{\psi}\in \widehat{G}$.
By definition and convention we conclude that
\,$\phi \in G_{0}$\,, \,$\rho_{\mu}\big(\tilde{\phi}\big)=(\,0\,,0\,)$\, and \,$\psi \in G$.
Then  \,$\tilde{\psi}\circ \tilde{\phi}\circ \tilde{\psi}^{-1}$\, is a lift
of \,${\psi}\circ {\phi}\circ {\psi}^{-1} \in G_{0}$\,. Since \,$\mu \in \mathcal{P}(G)$\,
it follows from Lemma \ref{lem:pact} that
$$\rho_{\mu}\big( \tilde{\psi}\circ \tilde{\phi}\circ \tilde{\psi}^{-1}\big)=
[\psi] \big( \rho_{\mu}(\tilde{\phi})\big)=(0\,,0). $$
 Consequently we have that
\,$\tilde{\psi}\circ \tilde{\phi}\circ \tilde{\psi}^{-1}\in \tilde{G}^{\mu}_{\mathcal{I}}$\,
and \,$\psi \circ \phi \circ \psi^{-1}\in {G}^{\mu}_{\mathcal{I}}$\,.
As a consequence of relation \eqref{rel:global}
the groups $\tilde{G}_{\mathcal I}$ and ${G}_{\mathcal I}$ are normal subgroups
of $\widehat{G}$ and $G$ respectively.
\end{proof}

\vskip10pt

The existence of lifts for normal subgroups can be interpreted in terms
of  $G$-invariant measures when
\,$L(\psi)\neq 0$\, for some \,$\psi \in G$.


\begin{pro}\label{pro:lifr}
\vskip10pt
Let \,$G$\, be a nilpotent subgroup of \,$\mathrm{Homeo}({\mathbb T}^{2})$\,
and let
\,$\psi \in G$\, such that  \,$L(\psi)\neq0$\,.
Fix a lift \,$\tilde{\psi}$\, of \,$\psi$\, and \,$\mu \in {\mathcal P}(G)$.
Consider a normal subgroup \,$H$\, of \,$G$\, with \,$H \subset G_{0}$\,.
Then \,$H$\, has a lift \,$\tilde{H}$\,
such that \,$\tilde{\psi} \, \tilde{H} \, \tilde{\psi}^{-1}=\tilde{H}$\,
if and only if \,$H \subset {G}_{\mathcal I}^{\mu}$\,.
In such a case we have
\,$\tilde{H} \subset \tilde{G}_{\mathcal I}^{\mu}$\,.

\end{pro}
\vskip5pt


Since \,$H$\, consists of isotopic to the identity homeomorphisms,
for a fixed \,$\psi\in G$\, we have that
\,$\tilde{\psi}\circ \tilde{\phi}\circ \tilde{\psi}^{-1}$\, does not depend on the choice of the lift
\,$\tilde{\psi}$\, of \,$\psi$\, when  \,$\phi\in {H}$.
Consequently, the condition
\,$\tilde{\psi} \, \tilde{H} \, \tilde{\psi}^{-1}=\tilde{H}$\,
does not depend on the choice of the lift \,$\tilde{\psi}$\,
of \,$\psi$.

\begin{proof}
First, if \,$H \subset G^{\mu}_{\mathcal{I}}$\, then
\,$\tilde{H}= \kappa^{-1}(H)$\, is a lift of \,$H$\, since
the natural projection
\,$\kappa:\tilde{G}^{\mu}_{\mathcal{I}}\rightarrow {G}^{\mu}_{\mathcal{I}}$\, is an isomorphism.
Consider any $\tilde{h} \in \tilde{H}$. We denote $h= \kappa(\tilde{h})$.
From Lemma \ref{lem:lifr} we know that the lift
\,$\tilde{\psi}\circ \tilde{h}\circ \tilde{\psi}^{-1}$\, of
\,${\psi}\circ {h}\circ {\psi}^{-1}$\, is in \,$\tilde{G}^{\mu}_{\mathcal{I}}$\, and we have
\,$\kappa(\tilde{\psi}\circ \tilde{h}\circ \tilde{\psi}^{-1})={\psi}\circ {h}\circ {\psi}^{-1}$.
Since \,$H$\, is normal in \,$G$\, we have \,${\psi}\circ {h}\circ {\psi}^{-1} \in H$\, and
then \,$\tilde{\psi}\circ \tilde{h}\circ \tilde{\psi}^{-1} \in \tilde{H}$\,
by definition of \,$\tilde{H}$. Consequently,
\,$\tilde{\psi} \, \tilde{H} \, \tilde{\psi}^{-1}=\tilde{H}$\,.

Now let us suppose that \,$H \subset G_{0}$\, admits a lift \,$\tilde{H}$\, such that
\,$\tilde{\psi} \, \tilde{H} \, \tilde{\psi}^{-1}=\tilde{H}$\,.
We want to prove that \,$\rho_{\mu}(\tilde{h})=(0\,,0)$\,
for any \,$\tilde{h} \in \tilde{H}$.
For this it suffices to show by induction on \,$j$\, that
\,$\rho_{\mu}\big(\tilde{h}\big)=(0\,,0)$\, for any \,$\tilde{h} \in \tilde{H}$\, such that \,$h \in H^{j}$\, where \,$H^{j}= Z^{(j)}(G) \cap H$\, and \,$j\geq0$\,.

The result is obvious for \,$j=0$. Suppose  it holds true for some \,$j\geq 0$\,
and consider \,$\tilde{h}\in \tilde{H}$\, with \,$h \in  H^{j+1}$. In that case the map
\,${\psi} \circ {h} \circ {\psi}^{-1} \circ {h}^{-1}$\,
 belongs to \,$H^{j}$\, for any \,$\psi \in G$.
Moreover we have that
\,$\tilde{\psi} \circ \tilde{h} \circ \tilde{\psi}^{-1} \circ \tilde{h}^{-1}$\, is contained in
\,$\tilde{H}$\, by hypothesis and we obtain\,:
\begin{align}
 (0\,,0) & = 
\rho_{\mu}\big(\tilde{\psi} \circ \tilde{h} \circ \tilde{\psi}^{-1} \circ \tilde{h}^{-1}\big)
=\rho_{\mu}\big(\tilde{\psi} \circ \tilde{h} \circ \tilde{\psi}^{-1}\big) +
\rho_{\mu}\big(\tilde{h}^{-1}\big) \notag\\
&   =
 [\psi]\big(\rho_{\mu}\big(\tilde{h}\big)\big) -
\rho_{\mu}\big(\tilde{h}\big)  =
 \big( [\psi] - Id \big)\big(\rho_{\mu}\big(\tilde{h}\big)\big).  \notag
\end{align}
The first equality is given by the induction hypothesis.
The others  are given by
Lemmas \ref{lem:mor} and
\ref{lem:pact}.
Since \,$L(\psi) \neq 0$\, we have from Lemma \ref{Lefs:spec}   that \,$1 \not \in \mathrm{spec}[\psi]$\,. Consequently, we deduce
\,$\rho_{\mu}\big(\tilde{h}\big)=(0\,,0)$\, for any \,$h \in H^{j+1}$. We obtain
\,$\tilde{H}\subset \tilde{G}^{\mu}_{\mathcal{I}}$\,.
\end{proof}


\vskip10pt

\begin{rem}
Proposition \ref{pro:lifr} admits an analogue for the case of a surface \,$S$\, of
genus \,$g \geq 2$. Let \,$G$\, be a nilpotent subgroup of
\,$\mathrm{Homeo}(S)$. Any element \,$\phi \in G$\, induces an element in
\,$\mathrm{Sp}(2g, {\mathbb Z})$\, that represents its action on the first
homology group.
The group \,$G_{0}$\, admits a canonical lift, the so called
identity lift. The existence of such lift implies that
all the homological rotation vectors of elements of \,$G_{0}$\,
with respect to measures in \,${\mathcal P}(G)$\,
are equal to \,$0$\, if \,$\cap_{\phi \in G} \mathrm{Ker}([\phi]-Id) = \{0\}$.
The proof is analogous to the proof of Proposition \ref{pro:lifr}.
\end{rem}


The existence of a lift for a subgroup does not depend on
the specific choice of an invariant measure.
This property is made explicit in the next proposition.

\vskip10pt
\begin{pro}
\label{pro:bffi}
Let \,$G$\, be a nilpotent subgroup of \,$\mathrm{Homeo}({\mathbb T}^{2})$.
Suppose there exists \,$\psi \in G$\,  with \,$L(\psi) \neq 0$\,.
Then we obtain
\,$\tilde{G}_{\mathcal I}=\tilde{G}_{\mathcal I}^{\mu}$\,
and \,$G_{\mathcal{I}} =G_{\mathcal I}^{\mu}$\, for any
\,$\mu \in {\mathcal P}(G)$.
Moreover \,$G_{\mathcal I}$\,
is  a finite index normal subgroup of \,$G_{0}$\,.

\end{pro}

 Let us remind the reader that \,$G_{\mathcal I}$\, is a normal subgroup of \,$G$. This has 
 been proved in Lemma \ref{lem:lifr}.

\begin{proof}

Let  \,$\mu \,, \nu \in {\mathcal P}(G)$\, and let \,$\psi\in G$\, with
\,$L(\psi) \neq 0$\,.
From Lemma \ref{lem:lifr} we have
\,$\tilde{\psi} \, \tilde{G}_{\mathcal I}^{\nu} \, \tilde{\psi}^{-1}=\tilde{G}_{\mathcal I}^{\nu}$\,.
Moreover, we know that \,$G^{\nu}_{\mathcal{I}}\subset G_{0}$\, is a normal subgroup of \,$G$. Then
we obtain
\,$\tilde{G}_{\mathcal I}^{\nu} \subset \tilde{G}_{\mathcal I}^{\mu}$\,
by Proposition \ref{pro:lifr}. Consequently, we deduce
\,$\tilde{G}_{\mathcal I}=\tilde{G}_{\mathcal I}^{\mu}$\, and
\,${G}_{\mathcal I}= {G}_{\mathcal I}^{\mu}$\,
for any \,$\mu \in {\mathcal P}(G)$.

Now we will  show that \,$G^{\mu}_{\mathcal{I}}$\, is a finite index normal subgroup of \,$G_{0}$\,. For this let us
 denote \,$H^{j}=Z^{(j)}(G) \cap G_{0}$\,.
We prove by induction on \,$j$\,
 that \,$H^{j} \cap G_{\mathcal I}^{\mu}$\, is a finite index normal subgroup
of \,$H^{j}$\, for any \,$j \geq 0$\,.

The result is clear for  \,$j=0$.
Suppose it holds true for some \,$j\geq0$\,.
In this case there exists \,$k\in \mathbb{Z}^{+}$\, such that
\,$g^{k}\in H^{j} \cap G_{\mathcal I}^{\mu}$\, for any \,$g\in H^{j}$. On the other hand the
$\mu$-rotation vector associated to  a lift of a homeomorphism in
\,$H^{j} \cap G_{\mathcal I}^{\mu}$\, is an element of \,${\mathbb Z}^{2}$\, by construction.
In particular, we have that \,$k\rho_{\mu}(\tilde{g})\in\mathbb{Z}^{2}$\, for any
lift \,$\tilde{g}$\, of   \,$g\in H^{j}$.

 Let us consider \,$h \in H^{j+1}$. We know that
\,$\psi \circ h \circ \psi^{-1} \circ h^{-1}$\,  is contained in \,$H^{j}$.
Now fix lifts \,$\tilde{\psi}$\, and \,$\tilde{h}$\,
of \,$\psi$\, and \,$h$\, respectively. Then we have
\,$k \rho_{\mu}(\tilde{\psi} \circ \tilde{h} \circ \tilde{\psi}^{-1} \circ \tilde{h}^{-1})
\in \mathbb{Z}^{2}$.
Moreover, it follows from
Lemmas \ref{lem:mor} and  \ref{lem:pact} that\,:
\[  k\rho_{\mu}(\tilde{\psi} \circ \tilde{h} \circ \tilde{\psi}^{-1} \circ \tilde{h}^{-1})=
k \big (\rho_{\mu}\big(\tilde{\psi} \circ
\tilde{h} \circ \tilde{\psi}^{-1}\big) -
\rho_{\mu}\big(\tilde{h}\big) \big) =
k\big([\psi]-Id \big)\big(\rho_{\mu}\big(\tilde{h}\big)\big) \in \mathbb{Z}^{2}. \]

 Since \,$L(\psi)\neq0$\, it follows from Lemma \ref{Lefs:spec} that
\,$1 \not \in \mathrm{spec}[\psi]$\,.  Then, we have that all the entries of the matrix
 \,$k' \big([\psi]-Id\big)^{-1}$\, are integer numbers
where \,$k'= \det ([\psi]-Id)$.
 Consequently we conclude that
$$kk' \rho_{\mu}\big( \tilde{h} \big) \in \mathbb{Z}^{2} \quad
\text{for any lift \,$\tilde{h}$\, of \,$h\in H^{j+1}$}.$$
This property guarantee us that the group morphism
$$h \in H^{j+1} \longrightarrow \ kk' \rho_{\mu}\big( \tilde{h} \big) \in
\mathbb{Z}^{2}/(kk')\mathbb{Z}^{2}$$
is well defined on \,$H^{j+1}$. Furthermore, its kernel is the set
$$\big \{ h\in H^{j+1} \ ; \ \rho_{\mu} \big( \tilde{h} \big) \in \mathbb{Z}^{2} \big \} =
H^{j+1} \cap G^{\mu}_{\mathcal{I}}$$
and we conclude that
\,$H^{j+1} / (H^{j+1} \cap G_{\mathcal I}^{\mu})$\,
is isomorphic to a subgroup of \,$\mathbb{Z}^{2}/(kk') \mathbb{Z}^{2}$.
 Therefore \,$H^{j+1} \cap G_{\mathcal I}^{\mu}$\,
has index at most \,$(k k')^{2}$\, in \,$H^{j+1}$\,  and the proof is complete.
\end{proof}

\begin{example*}
We define
\[ \widehat{H} = \{ T_{(a,b)} \ \ ; \ \  a\,,b \in {\mathbb Z}/2^{n} \} \quad
\mathrm{and} \quad \widehat{G} = \langle \widehat{H}, -Id \rangle \quad
\mathrm{where} \quad n\in \mathbb{Z}^{+}.
 \]
We consider the subgroups \,$H$\, and \,$G$\, of diffeomorphisms 
of \,${\mathbb T}^{2}$\, whose lifts belong to \,$\widehat{H}$\, and \,$\widehat{G}$\, respectively.
They are nilpotent subgroups of \,$\mathrm{Diff}^{\omega}({\mathbb T}^{2})$.
This is an example where \,$G_{\mathcal I}$\, is strictly contained in \,$G_{0}$.
Indeed \,$G$\, is a finite group such that
\,$G_{0}=H$\, and \,$G_{\mathcal I}=\{Id\}$\,
since \,$\rho_{\mu}(T_{(a,b)})=(a\,,b)$\, for any \,$T_{(a,b)}$-invariant
Borel probability measure \,$\mu$.
\end{example*}

The next theorem implies Theorem \ref{teo:main5}.
In its proof we use the following two lemmas. The proof of the first one was suggested to us by the referee.

\vskip10pt
\begin{lem}\label{lim:fix:spec}
Let \,$\tilde{\psi}\in \mathrm{Homeo}(\mathbb{R}^{2})$\, be a lift of 
\,$\psi\in \mathrm{Homeo}(\mathbb{T}^{2})$\, 
such that \,$1\notin \mathrm{spec}[\psi]$\,. Then \,$\mathrm{Fix}(\tilde{\psi})$\, is a compact set.
\end{lem}

\begin{proof}
Since \,$A:=[\psi]$\, is homotopic to \,$\psi$, any choice  of lifts 
\,$\tilde{A}$\, and \,$\tilde{\psi}$\, are a uniform finite distance
apart, i.e.
there exists \,$K$\, such that \,$\|\tilde{A}(x)-\tilde{\psi}(x)\|< K$\, for all \,$x\in\mathbb{R}^{2}$\, 
where \,$\|\cdot\|$\, is the usual norm in \,$\mathbb{R}^{2}$. Also, since 
\,$1\notin \text{spec}(A)$\, there exists \,$C>0$\, such that 
\,$\|A x -x\|\geq C\|x\|$\, for all \,$x$. It follows that \,$\|\tilde{\psi}(x)-x\|\geq C\|x\|-K$\, for all \,$x$\, and hence 
\,$\|\tilde{\psi}(x)-x\|>0$\, for all \,$x$\, with \,$\|x\|$\, sufficiently large. Therefore 
\,$\text{Fix}(\tilde{\psi})$\, lies in a bounded  subset of \,$\mathbb{R}^{2}$.
\end{proof}

\vskip10pt
\begin{lem}\label{lim:fix:spec:02}
 Let \,$\mathcal{G}$\, be a nilpotent subgroup of \,${\rm GL}(2\,,\mathbb{Z})$\,
such that there exists
\,$A \in \mathcal{G}$\, with \,$1 \not \in \mathrm{spec}(A)$.
 Then there exists $B \in \mathcal{G}$\, such that $1 \notin \mathrm{spec}(B)$,  
$\det (B)=1$ and $\langle B \rangle$ is a finite index normal subgroup of \, $\mathcal{G}$.
\end{lem}
\begin{proof}
 We have that the group \,$\mathcal{G}$\, has  the form\,:
\,${\{N^{n}\}}_{n \in {\mathbb Z}}$\ ,
\,${\{\pm N^{n}\}}_{n \in {\mathbb Z}}$\, for some \,$N\in \mathrm{GL}(2\,,\mathbb{Z})$\, or it is conjugated to
\,${\mathcal H}$\, by Lemma \ref{lem:nmcg}.
In the first case we have \,$1 \not \in \mathrm{spec}(N)$.
We claim \,$1 \not \in \mathrm{spec}(N^{2})$\,  if \,$N$\, is  orientation-reversing.
Otherwise $1  \in \mathrm{spec}(N^{2})$ implies
$\mathrm{spec}(N^{2}) = \{1\}$ since $\det (N^{2})=1$. We deduce 
$\mathrm{spec}(N) = \{-1,1\}$ as a consequence of $\det (N)=-1$. It contradicts $1 \not \in \mathrm{spec}(N)$. 
Thus we can  choose  \,$B=N$\,
if \,$N$\, is orientation-preserving and
\,$B=N^{2}$\, if \,$N$\, is orientation-reversing.
Analogously
in the case  \,$\mathcal{G}={\{\pm N^{n}\}}_{n \in {\mathbb Z}}$\,
we choose  
\,$B \in \{N^{2}, -N^{2} \}$\,
since $\det (N^{2})=1$ implies that 
$1 \in \mathrm{spec}(N^{2})$ and $1 \in \mathrm{spec}(-N^{2})$ can not hold simultaneously.
We choose  \,$B = - Id$ in the  case where \,$\mathcal{G}$\, is conjugated to \,${\mathcal H}$.
\end{proof}
\vskip5pt

\vskip10pt
\begin{teo}
\label{teo:fin}
Let \,$G$\, be a nilpotent subgroup of \,$\mathrm{Diff}^{1}({\mathbb T}^{2})$\,.
Suppose  there exists \,$\psi \in G$\, such that
\,$L(\psi) \neq 0$\,.
Then \,$\mathrm{Fix}(G_{\mathcal{I}}\,, \psi)\neq\emptyset$\, and the orbit of \,$p\in \mathrm{Fix}(G_{\mathcal{I}},\psi)$\, by  \,$G$\, is finite and it is contained in \,$\mathrm{Fix}(G_{\mathcal I})$.
\end{teo}

\begin{proof}

Let \,$[\,]: G \to \mathrm{MCG}({\mathbb T}^{2})$\, be the morphism of groups associating to each 
element of \,$G$\, its isotopy class in the mapping class group of the torus.
Since \,$\mathrm{ker}([\,]) = G_0$\, we obtain that \,$G / G_0$\, is isomorphic to the nilpotent subgroup 
\,$[G]$\, of \,$\mathrm{MCG}({\mathbb T}^{2})$.
Following Lemmas    \ref{lim:fix:spec:02}  and   \ref{Lefs:spec} 
there exists an orientation-preserving \,$\psi \in G$\, such that \,$L(\psi) \neq 0$\, and
\,$\langle [\psi] \rangle$\, is a finite index subgroup of \,$[G]$.
In particular \,$[\,]^{-1} \langle [\psi] \rangle$\, is a finite index normal subgroup of
\,$G/G_0$\, and hence \,$\langle G_0, \psi \rangle$\, is a finite index normal subgroup of \,$G$.

 Since \,$L(\psi) \neq 0$\,, let us  choose a lift
\,$\tilde{\psi}\in\mathrm{Diff}^{1}_{+}(\mathbb{R}^{2})$\, of \,$\psi$\, such that
\,$\mathrm{Fix}(\tilde{\psi}) \neq \emptyset$\,.
From Lemma  \ref{lim:fix:spec}
we conclude
that \,$\mathrm{Fix}(\tilde{\psi})$ is a non-empty compact subset of \,$\mathbb{R}^{2}$.

 Now, we denote \,$J= \langle G_{\mathcal I}\,, \psi \rangle$\, and
\,$\tilde{J}= \langle \tilde{G}_{\mathcal I}\,, \tilde{\psi} \rangle
\subset \mathrm{Diff}^{1}_{+}(\mathbb{R}^{2})$.
Since \,$G_{\mathcal{I}}$\, is a finite index subgroup of
\,$G_{0}$\, by Proposition \ref{pro:bffi}  and \,$G_{0}$\, is a normal
subgroup of \,$G$\, it follows that \,$J$\, is a finite index subgroup of \,$\langle G_0, \psi \rangle$.
Moreover, since \,$\langle G_0, \psi \rangle$\, is a finite index  subgroup of \,$G$\,  
we conclude that \,$J$\, is a finite index subgroup of \,$G$.

The group  \,$\tilde{G}_{\mathcal I}$\, is normal in
\,$\langle \tilde{G}_{\mathcal I}, \tilde{\psi} \rangle$\,
since it is normal in \,$\widehat{G}$\, as proved in Lemma \ref{lem:lifr}.
Then the derived group
\,$\langle \tilde{G}_{\mathcal I}, \tilde{\psi} \rangle'$\, of \,$\langle \tilde{G}_{\mathcal I}, \tilde{\psi} \rangle$\, is contained in
\,$\tilde{G}_{\mathcal I}$\, and hence
the natural projection given by
\,$\kappa: \langle \tilde{G}_{\mathcal I}, \tilde{\psi} \rangle' \rightarrow
\langle {G}_{\mathcal I}, {\psi} \rangle'$\, is an isomorphism.
We deduce that \,$\tilde{J}$\, is a nilpotent
subgroup of \,$\mathrm{Diff}_{+}^{1}({\mathbb R}^{2})$.
Moreover, we know that \,$\text{Fix}\big(\tilde{\psi}\big)$\, is a non-empty compact set.
Thus Theorem \ref{cor:plane} 
guarantees that
\,$\text{Fix}\big(\tilde{J}\,\big)\neq\emptyset$\, and
we conclude that  \,$\langle {G}_{\mathcal I}\,, {\psi} \rangle$\, has a global fixed point.

 On the other hand the group \,$\langle {G}_{\mathcal I}\,, {\psi} \rangle$\,
 is a finite index subgroup of \,$G$.
Let  \,$g_{1}\,,\ldots\,, g_{k} \in G$\, such that
$$G  =
\langle {G}_{\mathcal I}, {\psi} \rangle \cup  g_{1}\langle {G}_{\mathcal I}, {\psi} \rangle \cup
\ldots \cup
 g_{k} \langle {G}_{\mathcal I}, {\psi} \rangle$$
and let \,$p\in \text{Fix}(\langle {G}_{\mathcal I}\,, {\psi} \rangle)$\,.
The orbit \,$\mathcal{O}$\, of \,$p$\, by \,$G$\, is given by
$$\mathcal{O}=\big\{p\,,g_{1}(p), \ldots, g_{k}(p) \big\}$$
where \,$\phi\big(g_{i}(p)\big)=g_{i}\circ g_{i}^{-1} \circ \phi \circ g_{i}(p)=g_{i}(p)$\,
for every \,$\phi\in G_{\mathcal{I}}$\, since \,$G_{\mathcal{I}}$\, is normal in \,$G$.
Consequently, \,${\mathcal O}$\, is contained in
\,$\mathrm{Fix}(G_{\mathcal I})$.
\end{proof}

\vskip30pt
\section{
Replacing \,$\mathbb{T}^{2}$\, with the compact annulus, the Klein bottle or
the M\"obius strip
}
\vskip10pt

In this last section we remark
that Theorem \ref{teo:main5} is also true for the circle \,$\mathbb{S}^{1}$, the compact annulus \,$\mathbb{S}^{1}\!\times[0\,,1]$\,,
the Klein bottle and the compact M\"obius strip.

 For this, let us remind the reader that a homeomorphism of 
\,$\mathbb{S}^{1}$\, has  non-zero Lefschetz number if and only
if it is orientation-reversing.

\vskip10pt
\begin{pro}\label{pro:s1}
Theorem \ref{teo:main5} stays true for nilpotent groups of homeomorphisms
when we replace \,$\mathbb{T}^{2}$\, with \,$\mathbb{S}^{1}$.
\end{pro}

If the elements of the nilpotent group \,$G$\, are \,$\mathbb{S}^{1}$-diffeomorphisms then Theorem \ref{teo:main5} implies the above proposition. It suffices to pass from the 
nilpotent group \,$G$\, to the nilpotent group \,$G\times G$\, where the element \,$(g_{1}\,,g_{2})\in G\times G$\, is given by
\[  (g_{1}\,,g_{2})(x\,,y)=\big(g_{1}(x)\,,g_{2}(y)\big)\in 
\mathbb{S}^{1} \times \mathbb{S}^{1}=\mathbb{T}^{2} \]
for all \,$x\,,y \in \mathbb{S}^{1}$.  
There exists an  orientation-reversing $g \in G$ by hypothesis. We can apply 
Theorem \ref{teo:main5} since $[g,g] = - Id$. Indeed the 
Lefschetz number of $(g,g)$ is non-zero by Lemma
\ref{Lefs:spec}.

When the elements of \,$G$\, are \,$\mathbb{S}^{1}$-homeomorphisms  we have the following proof.

\begin{proof}[Proof of Proposition \ref{pro:s1}]

Let \,$\psi$\, be an orientation-reversing homeomorphism of a subgroup \,$G$\, of
\,$\mathrm{Homeo}({\mathbb S}^{1})$. The fixed point set of \,$\mathrm{Fix}(\psi)$\, contains exactly
two points.
Analogously to Lemma \ref{lem:lifr} the group \,$G_{\mathcal I}$\,
of elements with \,$0 \in {\mathbb R}/{\mathbb Z}$\, rotation number is a  normal subgroup
of the group \,$G_{0}$\, of orientation-preserving elements of \,$G$.
We can proceed as in Proposition \ref{pro:bffi} to show that \,$G_{\mathcal I}$\, is a finite
index subgroup of \,$G_0$. 
More precisely the rotation numbers associated to elements of \,$H^{j} :=Z^{(j)}(G) \cap G_{0}$\,
are contained in \,$({\mathbb Z}/2^{j})/ {\mathbb Z}$\, 
and \,$H^{j} \cap G_{\mathcal I} =  \{ \phi \in H^{j} : \mathrm{Fix}(\psi) \subset \mathrm{Fix}(\phi) \}$\,
for any \,$j \geq 0$. It is obvious for \,$j=0$. 
Suppose it holds for some \,$j \geq 0$. Given \,$\phi \in H^{j+1}$\, the element 
\,$\eta:= \phi \circ \psi \circ \phi^{-1} \circ \psi^{-1}$\, belongs to \,$H^{j}$\, and
has rotation number \,$2 \rho(\phi)$\, where \,$\rho (\phi)$\, is the rotation number of \,$\phi$. Thus 
\,$\rho (\phi)$\, belongs to \,$({\mathbb Z}/2^{j+1})/ {\mathbb Z}$.
Moreover if \,$\phi \in G_{\mathcal I}$\, then \,$\eta$\, belongs to \,$H^{j} \cap  G_{\mathcal I}$.
Since \,$\phi \circ \psi \circ \phi^{-1} = \eta \circ \psi$\,,  we deduce
\,$\phi (\mathrm{Fix}(\psi)) = \mathrm{Fix}(\eta \circ \psi)$.
The induction hypothesis implies \,$\mathrm{Fix}(\psi) \subset \mathrm{Fix}(\eta \circ \psi)$.
Since both \,$\psi$\, and \,$\eta \circ \psi$\, are orientation-reversing, their fixed point sets contain
exactly two points. As a consequence we obtain \,$\mathrm{Fix}(\eta \circ \psi) =\mathrm{Fix}(\psi)$\, and
then \,$\phi (\mathrm{Fix}(\psi)) = \mathrm{Fix}(\psi)$\, for any 
\,$\phi \in H^{j+1} \cap G_{\mathcal I}$.
The rotation number of \,$\phi$\, is \,$0$\, and thus any periodic orbit of \,$\phi$\, 
consists of fixed points.
Hence \,$\mathrm{Fix}(\psi) \subset \mathrm{Fix}(\phi)$\, for any 
\,$\phi \in H^{j+1} \cap G_{\mathcal I}$.

By applying the previous result for  the nilpotency class \,$j$\, of \,$G$\, we obtain that 
\,$G_{0}/G_{\mathcal I}$\, is isomorphic to a subgroup of the finite group \,$({\mathbb Z}/2^{j})/ {\mathbb Z}$. 
In particular \,$G_{\mathcal I}$\, is a finite index normal subgroup of \,$G_0$. 
Moreover we get  \,$\mathrm{Fix}(\langle G_{\mathcal I}, \psi \rangle)= \mathrm{Fix}(\psi)$.

Since \,$G_0$\, is a finite index subgroup of \,$G$,  we have that
\,$\langle G_{\mathcal I}, \psi \rangle$\, is a finite index subgroup of \,$G$.
Hence there exists a finite orbit of \,$G$\, whose intersection with \,$\mathrm{Fix}(\psi)$\, is non-empty.
\end{proof}

Before proving  Theorem \ref{teo:main5} for the compact annulus let us remark that a
homeomorphism of
\,$\mathbb{S}^{1}\!\times[0\,,1]$\, has  Lefschetz number different from zero if and only
if it changes the orientation of the generator of the first singular homology group of the annulus.

\vskip10pt
\begin{teo}\label{teor:annulus}

Let \,$\mathcal{N}$\, be a nilpotent subgroup of \,$\mathrm{Diff}^{1}(\mathbb{S}^{1}\!\times[0\,,1])$\,. If
\,$\mathcal{N}$\, has some element with non-zero Lefschetz number
then \,$\mathcal{N}$\, has  a finite orbit. More precisely$\,:$
\begin{itemize}
\item[$(i)$]
$\mathcal{N}$\, has finite orbits in
the annulus boundary when
there exists an element of \,$\mathcal{N}$\, with non-zero Lefschetz number that  
leaves invariant each of the connected components of the annulus boundary$\,;$

\item[$(ii)$]
$\mathcal{N}$\, has  finite orbits in the interior of the annulus 
if no connected component of the annulus boundary is ${\mathcal N}$-invariant.

\end{itemize}

\end{teo}
\vskip5pt

\begin{proof}

Consider the subgroup ${\mathcal N}'$ of ${\mathcal N}$ 
whose elements leave invariant each one of the connected components of 
the boundary $\partial A$ of the annulus $A:= \mathbb{S}^{1}\!\times[0\,,1]$.
Let us consider the proposed two cases.

\vglue5pt

\noindent
{\sc Case $(i)$\,:}  
The group ${\mathcal N}'$ has finite orbits in both connected components of $\partial A$
by Proposition \ref{pro:s1}. Since ${\mathcal N}'$  is a subgroup of ${\mathcal N}$ of index at most $2$,
it follows that ${\mathcal N}$ has finite orbits contained in $\partial A$. 
The result holds true even when the group \,$\mathcal{N}$\, 
is contained in \,$\mathrm{Homeo}(\mathbb{S}^{1}\!\times[0\,,1])$.

\vglue5pt

\noindent
{\sc Case $(ii)$\,:} 
To detect a finite orbit in the interior of the annulus we  consider the double
\,$\mathbb{T}^{2}$\, of the annulus and the double of each element of \,$\mathcal{N}$. 
With this construction we obtain a nilpotent subgroup
\,$G$\, of
\,$\mathrm{Homeo}(\mathbb{T}^{2})$.
We claim that there exists $f_0 \in {\mathcal N} \setminus {\mathcal N}'$ such that 
$L(f_0) \neq 0$. Otherwise $L(f)$ vanishes for any $f \in {\mathcal N} \setminus {\mathcal N}'$
and there exists $f_1 \in {\mathcal N}'$ such that $L(f_1) \neq 0$ by hypothesis. 
We choose any $g_1 \in  {\mathcal N} \setminus {\mathcal N}'$. We obtain
$L(f_1 g_1) \neq 0$ since $f_1 g_1$ changes the orientation in $H_{1} (A, {\mathbb Z})$.
Since $f_1 g_1$ belongs to ${\mathcal N} \setminus {\mathcal N}'$ the result is proved.
The class in the mapping class group of the double $\psi \in G$ of the map $f_0$
is given by the matrix
$$[\psi]=\left[\begin{array}{cc}-1 & n \\0 & -1\end{array}\right] \quad
\textrm{for some} \quad n\in\mathbb{Z}.$$
We have  \,$L(\psi)=4$\, since
\,$L(\psi) = \det (Id - [\psi])$\, by \cite{Brooks-Brown-Pak-Taylor}.

At this point we can not apply Theorem \ref{teo:main5} in a straightforward way
to obtain a finite orbit because the elements of \,$G$\, defined in
the double of the annulus are not necessarily of class \,$C^{1}$ along the boundary of the annulus in \,$\mathbb{T}^{2}$.

Nevertheless,  following another referee's suggestion  we can  overcome this point  using a recent result of 
K. Parkhe in  \cite{Kiran}. He proves (cf. Theorem 5 and Remark 6 in section 2) there exists a 
homeomorphism of the annulus, supported in a small neighborhood of its boundary such that: the conjugate
of the elements of \,$G$\, by this homeomorphism gives a \,$C^{1}$-diffeomorphism that can be glued with
itself  to obtain a \,$C^{1}$-\,diffeomorphisms in the double \,$\mathbb{T}^{2}$ of the annulus.

Of course the \,$C^{1}$-\,diffeomorphism corresponding to the map \,$\psi$\, stays with non-zero Lefschetz number in \,$\mathbb{T}^{2}$. Consequently, since \,$\psi$\, has no fixed point in the boundary of the annulus it follows from Theorem \ref{teo:fin} that \,$\mathcal{N}$\, has a finite orbit in the interior of the annulus.
\end{proof}

\begin{rem}
We also can prove the above result repeating the same arguments presented in the proof of Theorem \ref{teo:fin}.
Let us remark that we use the \,$C^{1}$ differentiability in the proof of such theorem
just to guarantee a global fixed point for
\,$\langle \tilde{G}_{\mathcal{I}}\, , \tilde{\psi} \rangle \subset \mathrm{Diff}^{1}_{+}(\mathbb{R}^{2})$\, via Theorem \ref{cor:plane}. But we can do this in another way.

Let us fix the universal covering map \,$\pi:\mathbb{R}^{2}\rightarrow\mathbb{T}^{2}$\,
such that the restrictions of \,$\pi$\, to the strips
\,$\mathbb{R} \times[0\,,1/2]$\,
and \,$\mathbb{R} \times[1/2\,,1]$\,
are the universal covering maps corresponding to the two copies of the annulus in the double.

The elements of \,$\tilde{G}_{\mathcal{I}}$\,
have a trivial rotation vector with respect to the Borel probability measures invariant by the group \,$G$. Hence, the strips \,$\mathbb{R} \times[0\,,1/2]$\,
and \,$\mathbb{R} \times[1/2\,,1]$\,
are invariant by all \,$\tilde{\phi}\in \tilde{G}_{\mathcal{I}}$.

 On the other hand we know that, in the double, the map \,$\psi$\, has a fixed point in the interior of each copy of the annulus  since \,$L(\psi)\neq0$\,  and \,$\psi$\,
 permutes the connected components of the annulus. Now, consider a lift \,$\tilde{\psi}$\, having a fixed point in the strip \,$\mathbb{R} \times[0\,,1/2]$\,.
 Then, from Lemma \ref{lim:fix:spec} we know that \,$\mathrm{Fix}(\tilde{\psi})$\, is a non-empty compact set. Moreover we have that the
strip \,$\mathbb{R} \times[0\,,1/2]$\, is also invariant by \,$\tilde{\psi}$. Consequently,
we can apply Theorem \ref{cor:plane} to the nilpotent group given by the restriction of the elements of
\,$\langle \tilde{G}_{\mathcal{I}}\, , \tilde{\psi} \rangle $\, to the strip
 \,$\mathbb{R} \times(0\,,1/2)$\, obtaining a global fixed point for the group
\,$\langle \tilde{G}_{\mathcal{I}}\, , \tilde{\psi} \rangle $\,. Following the end of the proof of Theorem \ref{teo:fin} we conclude the existence of a finite orbit for the
group \,$G$\, and consequently, the group \,$\mathcal{N}$\, has a finite orbit in the interior of the annulus.

\end{rem}
\vskip4pt

 In what follows let  \,$\Pi:\mathbb{T}^{2} \rightarrow \mathbb{K}^{2}$\, be the
\,$2$-fold orientation covering map of the Klein bottle \,$\mathbb{K}^{2}$\, by \,$\mathbb{T}^{2}$\, and let us denote by \,$\sigma$\, the non-trivial lift  of the identity map by \,$\Pi$.
Given a subgroup  \,$G$ of $\mathrm{Homeo(\mathbb{K}^{2})}$, let us denote by
\,$\widehat{G} \subset \mathrm{Homeo(\mathbb{T}^{2})}$\, the subgroup of all lifts of elements of
\,$G$.

 Given an element \,$\phi \in G$\, its distinct lifts to \,$\mathbb{T}^{2}$\, by \,$\Pi$\, are
\,$\tilde{\phi}$\, and \,$\sigma \circ \tilde{\phi}$. Hence we have that the lifts of
\,$\phi \in G$\, by \,$\Pi$\, commute with the covering transformations since the
lifts
\,$\tilde{\phi} \circ \sigma$\, and \,$\sigma \circ \tilde{\phi}$\, are different from
\,$\tilde{\phi}$ and then equal.
This property implies that if
\,$G$\, is a nilpotent group with nilpotent class \,$n\in\mathbb{Z}^{+}$\, then
\,$\widehat{G}$\, is also a nilpotent group, with 
nilpotent class at most $n+1$.

 Moreover, if \,$\psi \in \mathrm{Homeo}(\mathbb{K}^{2})$\,  then
the Lefschetz numbers of the lifts of \,$\psi$\, by \,$\Pi$\, are
$2 L(\psi)$ and $0$.
Consequently, if \,$G$\, has an element with non-trivial Lefschetz number
then \,$\widehat{G}$\, has the same property and we obtain
the following corollary of Theorem \ref{teo:main5}.

\vskip10pt
\begin{cor}
Let \,$G$\, be a nilpotent subgroup of \,$\mathrm{Diff}^{1}({\mathbb K}^{2})$\,.
If \,$G$\, has some element whose  Lefschetz number
is different from zero
then \,$G$\, has a finite orbit.
\end{cor}
\vskip5pt

Repeating for the M\"obius strip the above arguments presented in this section for the annulus and for the Klein bottle we conclude the following corollary of  Theorem \ref{teor:annulus}.

\vskip10pt
\begin{cor}
Let \,$\mathcal{N}$\, be a nilpotent subgroup of the group of all \,$C^{1}$-diffeomorphisms of the compact
M\"obius strip. If
\,$\mathcal{N}$\, has some element \,$f$\, with \,$L(f)\neq 0$\, then
\,$\mathcal{N}$\, has  finite orbits in the boundary and in the interior of the M\"obius strip.
\end{cor}

The $C^{1}$ hypothesis is not necessary to find finite orbits in the boundary
of the M\"obius strip. Such a boundary has one connected component that is
homeomorphic to a circle. Hence there exist finite orbits in the boundary by Proposition \ref{pro:s1}.
Finally let us remind the reader
that the non-trivial covering transformation \,$\sigma$\, associated to the \,$2$-fold
orientation covering  of the compact M\"obius strip by the compact annulus,
permutes the connected components of the
boundary of the annulus. In that case,  one of the lifts
\,$\tilde{\psi}\,,\sigma \circ \tilde{\psi}$\, permutes such components.
There exists an interior finite orbit by Theorem \ref{teor:annulus}.

\vglue20pt
\appendix
\section{Proof of Lemma \ref{lem:nmcg}}

Let \,${\mathcal G}$\, be a nilpotent subgroup of \,$\mathrm{GL}(2\,,{\mathbb Z})$.
Consider the natural mapping \,$\tau: \mathrm{GL}(2\,,{\mathbb C}) \to \mathrm{PGL}(2\,,{\mathbb C})$.
We denote \,$G_{+}=\tau({\mathcal G} \cap \mathrm{SL}(2\,,{\mathbb Z}))$\,
and \,$G=\tau({\mathcal G})$.
It suffices to show that either \,$G$\, is cyclic or
\,${\mathcal G}$\, is conjugated to \,${\mathcal H}$.

If \,$G_{+}$\, is trivial then \,$G$\, is a group of cardinality at most \,$2$.
We suppose from now on that \,$G_{+}$\, is non-trivial.

The group \,$G_{+}$\, is a nilpotent Fuchsian group. Since its center is non-trivial all
elements share the same fixed point set. Hence \,$G_{+}$\, is a cyclic group 
\,$\langle \alpha \rangle$\,
(cf. \cite[Theorems 2.3.2 and 2.3.5]{Katok-fuchs} for further details).

Since \,$G$\, is nilpotent, its subset \,$\mathrm{Tor}(G)$\, of finite order elements is a normal subgroup
of \,$G$\, \cite[Theorem 16.2.7]{Karga}.
Suppose that   \,$\mathrm{Tor}(G)$\,
is non-trivial. We claim that \,$G$\, is finite.
Since every non-trivial normal subgroup of a nilpotent group \,$H$\, contains a non trivial element of
the center of \,$H$, the group \,$\mathrm{Tor}(G) \cap Z(G)$\, is non-trivial 
\cite[Theorem 16.2.3]{Karga}.
If \,$\mathrm{Tor}(G) \cap G_{+} \neq \{Id\}$\, then
\,$G_{+}$\, is finite and so is \,$G$. In the remaining case there exists a non-trivial
element \,$\beta \in \mathrm{Tor}(G) \cap Z(G)$\, that necessarily belongs to \,$G \setminus G_{+}$.
Consider an element \,$B \in {\mathcal G}$\, such that \,$\tau (B)=\beta$.
The matrix \,$B$\, has finite order; thus its eigenvalues are roots of the unit.
The matrix \,$B$\, does not have a non-real eigenvalue \,$\lambda$\, since otherwise
\,$\det (B)= \lambda \overline{\lambda} =1$\, and this contradicts \,$\tau (B) \not \in G_{+}$.
Hence the eigenvalues of \,$B$\, are necessarily \,$1$\, and \,$-1$\, and we can diagonalize
\,$B$\, up to conjugation by a matrix in \,$\mathrm{GL}(2\,,{\mathbb Q})$.
Since \,$\tau (B) \in Z(G)$\, we obtain either \,$[\,C\hskip1pt,B\,]=Id$\, or \,$[\,C\hskip1pt,B\,]=-Id$\, for any
\,$C \in {\mathcal G}$.
We denote \,${\mathcal G}_{1}= \big\{ C \in {\mathcal G} \ ; \ [\,C\hskip1pt,B\,]=Id \big\}$\,;
it is a normal subgroup of \,${\mathcal G}$\, of index at most \,$2$.
The eigenvalue associated to an eigenvector in
\,${\mathbb Q}^{2} \setminus \{(0\,,0)\}$\, of a matrix \,$E$\, in \,$\mathrm{GL}(2\,,{\mathbb Z})$\,
is always $1$ or $-1$. We deduce that
\,$C$\, is diagonal in the base diagonalizing \,$B$\,
with entries in \,$\{1\,,-1\}$\, for any \,$C \in {\mathcal G}_{1}$.
Hence the cardinality of \,${\mathcal G}_{1}$\, is less or equal than \,$4$\, and \,$G$\, is finite.

Since
\,$G_{+}$\, is a normal subgroup of \,$G$,
there exists a non-trivial element \,$\gamma \in G_{+} \cap Z(G)$.
Suppose \,$\gamma$\, is parabolic (i.e.  $\#(\mathrm{Fix}(\gamma))=1$).
The restriction of \,$\gamma$\, to the circle \,${\mathbb R} \cup \{\infty\}$\, has exactly a fixed
point and all the remaining orbits are infinite.
Let us show that \,$\gamma$\, does not commute with any orientation-reversing homeomorphism \,$\eta$\,
of the circle \,${\mathbb R} \cup \{\infty\}$\, by contradiction.
Indeed \,$\gamma \circ \eta = \eta \circ \gamma$\, implies \,$\gamma (\mathrm{Fix}(\eta)) =  \mathrm{Fix}(\eta)$.
Moreover \,$\mathrm{Fix}(\eta)$\, contains exactly two points if \,$\eta_{|{\mathbb R} \cup \{\infty\}}$\, is
orientation-reversing. We deduce that the two points of \,$\mathrm{Fix}(\eta)$\, 
have finite orbits for \,$\gamma$\,,
obtaining a contradiction since there is just one point whose orbit by \,$\gamma$\,
is finite. Therefore
the groups \,$G_{+}$\,
and \,$G$\, coincide and then \,$G$\, is cyclic.

Suppose that \,$\gamma$\, is hyperbolic, it has \,$2$\, fixed points
in \,${\mathbb R} \cup \{\infty\}$\, and no other finite orbit.
We already know that \,$\mathrm{Fix}(\gamma) = \mathrm{Fix}(\eta)$\,
for any \,$\eta \in G_{+} \setminus \{Id\}$. Given \,$\eta \in G \setminus G_{+}$\,
the properties  \,$\gamma (\mathrm{Fix}(\eta)) = \mathrm{Fix}(\eta)$\, and
\,$\sharp (\mathrm{Fix}(\eta))=2$\, imply \,$\mathrm{Fix}(\eta) \subset \mathrm{Fix}(\gamma)$\, and then
\,$\mathrm{Fix}(\eta) = \mathrm{Fix}(\gamma)$. We obtain 
\,$\mathrm{Fix}(\gamma) = \mathrm{Fix}(\eta)$\,
for any \,$\eta \in G \setminus \{Id\}$.
Fix \,$p_{0} \in \mathrm{Fix}(\gamma)$\,,
we define 
\,$\zeta: G \to {\mathbb R}^{*}$\, and $|\zeta|: G \to {\mathbb R}^{+}$ as
\,$\zeta (\eta) = \eta'(p_{0})$ and $|\zeta| (\eta) = |\eta'(p_{0})|$ respectively.
The map \,$\zeta$\, is injective since a M\"obius transformation 
that has two different fixed points (the elements of \,$\mathrm{Fix}(\gamma)$) and whose 
multiplicator at one of them is equal to \,$1$, it is the identity map.
Since \,$\zeta (G_{+})$\, is a cyclic subgroup of \,${\mathbb R}^{+}$\, we conclude that
\,$|\zeta| (G)$\, is a discrete closed subgroup of \,${\mathbb R}^{+}$. 
In particular \,$|\zeta| (G)$\, is cyclic. Then there exists 
\,$\eta_0 \in G$\, such that \,$|\eta_{0}'(p_{0})|$\, generates \,$|\zeta| (G)$.
We claim \,$G = \langle \eta_0 \rangle$\, and in particular that \,$G$\, is cyclic. 
Otherwise there exists an element \,$\eta_1 \in G$\,
such that \,$\zeta (\eta_1) = -1$. Then \,$\eta_1$\, is an element of order \,$2$\, of \,$G$\, since 
\,$\zeta$\, is injective.  
Since \,$\eta_1 \in \mathrm{Tor}(G) \setminus \{Id\}$,   
\,$G$\, is finite. 
This contradicts that \,$\gamma$\, is hyperbolic.

The unique remaining case corresponds to the situation where \,$\gamma$\, is of finite order.
The group \,$G$\, is a finite nilpotent group of  M\"obius transformations.
The finite groups of orientation-preserving homeomorphisms of the Riemann sphere are isomorphic to 
a cyclic group \,$C_n$\,, a dihedral group \,$D_n$\,, \,$A_4$\,, \,$S_4$\, or \,$A_5$\, \cite[Theorem 2.6.1]{Shurman}.
The groups \,$A_4$\,, \,$S_4$\, or \,$A_5$\, are not nilpotent. Moreover \,$D_n$\, is nilpotent if and only if \,$n$\, is 
a power of \,$2$. Thus 
if \,$G$\, is not cyclic then it is a dihedral group \,$D_{2^{m}}$\, with \,$2^{m+1}$\, elements
for some \,$m \in {\mathbb N}$.
It is easy to see that the periodic elements of
\,$\tau (\mathrm{GL}(2\,,{\mathbb Z}))$\, have order \,$1$\,, $2$\, or \,$3$.
Since \,$D_{2^{m}}$\, contains a cyclic group with \,$2^{m}$\, elements we deduce that
\,$G$\, is the group \,$D_{2}$.
Let \,$A$\, be a matrix such that \,$\tau(A) = \alpha$.
Since \,$\alpha$\, belongs to \,$G_{+}$\, and has order \,$2$\,, we deduce
\,$A \not \in \{Id \,, - Id\}$\,, \,$A^{2} \in \{Id\,, - Id\}$\, and \,$\det (A)=1$.
These properties imply
\,$A^{2}=-Id$\, and \,$\mathrm{spec}(A)=\{i\,,-i\}$.
The kernel of \,$\tau_{|{\mathcal G}}$\, is equal to \,$\{Id \,, -Id\}$\,,
in particular \,${\mathcal G}$\, has \,$8$\, elements.
Consider a matrix \,$B \in {\mathcal G}$\, such that \,$\tau (B) \not \in G_{+}$.
Analogously as in the fourth paragraph the matrix \,$B$\,
satisfies \,$\mathrm{spec}(B) =\{1\,,-1\}$\, and it is diagonalizable by
a change of coordinates in \,$\mathrm{GL}(2\,,{\mathbb Q})$.
Moreover \,$A$\, does not commute with \,$B$\, since otherwise
\,$\mathrm{spec}(A) \subset {\mathbb R}$.
Since \,$[\,A\,,B\,]=-Id$\, we obtain\,:
\[ A =
\left(
\begin{array}{cc}
0 & a \\
-a^{-1} & 0 \\
\end{array}
\right) \quad \text{and} \quad
B = \left(
\begin{array}{rr}
1 & 0 \\
0 & -1 \\
\end{array}
\right) \quad  \text{for some \,$a \in {\mathbb Q}$.} \]
Up to a further change of coordinates we can suppose that
\,$a=1$. This implies \,${\mathcal G}={\mathcal H}$.
Since \,$A$\, has order \,$4$, \,$B$\, has order \,$2$\, and \,$B A B^{-1} = A^{-1}$,
the group \,${\mathcal H}$\, is isomorphic to \,$D_{4}$.

\end{document}